\documentclass[reqno,11pt]{amsart}
\usepackage{amsmath, latexsym, amsfonts, amssymb, amsthm, amscd}
\usepackage{multirow}

\usepackage{xcolor}

\usepackage{graphics,epsf,psfrag}
\numberwithin{equation}{section}

\newtheorem{theorem}{Theorem}[section]
\newtheorem{lemma}[theorem]{Lemma}
\newtheorem{proposition}[theorem]{Proposition}

\newtheorem{rem}[theorem]{Remark}

\DeclareMathOperator{\p}{\mathbb{P}}

\newcommand{\ind}{\mathbf{1}}

\newcommand{\R}{\mathbb{R}}

\newcommand{\N}{\mathbb{N}}

\newcommand{\cB}{{\ensuremath{\mathcal B}} }

\newcommand{\cP}{{\ensuremath{\mathcal P}} }

\newcommand{\cC}{{\ensuremath{\mathcal C}} }

\newcommand{\bP}{{\ensuremath{\mathbf P}} }
\newcommand{\bE}{{\ensuremath{\mathbf E}} }

\DeclareMathSymbol{\leqslant}{\mathalpha}{AMSa}{"36} % nicer `smaller or equal'
\DeclareMathSymbol{\geqslant}{\mathalpha}{AMSa}{"3E} % nicer `larger or equal'
\DeclareMathSymbol{\eset}{\mathalpha}{AMSb}{"3F}     % nicer `emptyset'
\newcommand{\dd}{\,\text{\rm d}}             % a straight d for differentials
\newcommand{\bbE}{{\ensuremath{\mathbb E}} }

\newcommand{\bbN}{{\ensuremath{\mathbb N}} }

\newcommand{\bbP}{{\ensuremath{\mathbb P}} }

\newcommand{\bbR}{{\ensuremath{\mathbb R}} }

\newcommand{\gd}{\delta}
\newcommand{\gp}{\varphi}

\newcommand{\gs}{\sigma}

\newcommand{\norm}[1]{\left\lVert#1\right\rVert}

\makeatletter
\def\captionfont@{\footnotesize}
\def\captionheadfont@{\scshape}

\long\def\@makecaption#1#2{%
  \vspace{2mm}
  \setbox\@tempboxa\vbox{\color@setgroup
    \advance\hsize-6pc\noindent
    \captionfont@\captionheadfont@#1\@xp\@ifnotempty\@xp
        {\@cdr#2\@nil}{.\captionfont@\upshape\enspace#2}%
    \unskip\kern-6pc\par
    \global\setbox\@ne\lastbox\color@endgroup}%
  \ifhbox\@ne % the normal case
    \setbox\@ne\hbox{\unhbox\@ne\unskip\unskip\unpenalty\unkern}%
  \fi
  \ifdim\wd\@tempboxa=\z@ % this means caption will fit on one line
    \setbox\@ne\hbox to\columnwidth{\hss\kern-6pc\box\@ne\hss}%
  \else % tempboxa contained more than one line
    \setbox\@ne\vbox{\unvbox\@tempboxa\parskip\z@skip
        \noindent\unhbox\@ne\advance\hsize-6pc\par}%
\fi
  \ifnum\@tempcnta<64 % if the float IS a figure...
    \addvspace\abovecaptionskip
    \moveright 3pc\box\@ne
  \else % if the float IS NOT a figure...
    \moveright 3pc\box\@ne
    \nobreak
    \vskip\belowcaptionskip
  \fi
\relax
}
\makeatother
\def\writefig#1 #2 #3 {\rlap{\kern #1 truecm
\raise #2 truecm \hbox{#3}}}

\title[Scaling limits of interacting diffusions on Erd\H os-R\'enyi graphs]{A Law of Large Numbers and  Large Deviations   \\ for
interacting diffusions on Erd\H os-R\'enyi graphs}

\author{Fabio Coppini}
\address{
  Universit\'e Paris Diderot, Sorbonne Paris Cit\'e,   Laboratoire de Probabilit{\'e}s Statistique et Mod\'elisation, UMR 8001,
            F- 75205 Paris, France
}

\author{Helge Dietert}
\address{CNRS, Universit\'e Paris Diderot and Sorbonne Universit\'e,   Institut de Math\'ematiques de Jussieu - Paris Rive Gauche,
            F- 75205 Paris, France
}

\author{Giambattista Giacomin}
\address{
  Universit\'e Paris Diderot, Sorbonne Paris Cit\'e,   Laboratoire de Probabilit{\'e}s Statistique et Mod\'elisation, UMR 8001,
            F- 75205 Paris, France
}

\begin{document}

\begin{abstract}
  We consider a class of particle systems described by differential
  equations (both stochastic and deterministic), in which the
  interaction network is determined by the realization of an Erd\H
  os-R\'enyi graph with parameter $p_n\in (0, 1]$, where $n$ is the
  size of the graph (i.e., the number of particles). If $p_n\equiv 1$
  the graph is the complete graph (mean field model) and it is well
  known that, under suitable hypotheses, the empirical measure
  converges as $n\to \infty$ to the solution of a PDE: a McKean-Vlasov
  (or Fokker-Planck) equation in the stochastic case, or a
  Vlasov equation in the deterministic one. It has already been shown
  that this holds for rather general interaction networks, that
  include Erd\H os-R\'enyi graphs with $\lim_n p_n n =\infty$, and
  properly rescaling the interaction to account for the dilution
  introduced by $p_n$. However, these results have been proven under
  strong assumptions on the initial datum which has to be
  \emph{chaotic}, i.e. a sequence of independent identically
  distributed random variables. The aim of our contribution is to
  present results -- Law of Large Numbers and Large Deviation
  Principle -- assuming only the convergence of the empirical measure
  of the initial condition.
\end{abstract}

\maketitle

\section{Introduction}

\subsection{Basic notations, the models and a first look at the main question}
Large systems of interacting diffusions with mean field
type interactions have been an important research topic in the
mathematical community at least since the 60's. The program of
identifying the emerging behavior for $n\to \infty$,
where $n$ is the number of interacting \emph{units}, has been
fully developed under suitable regularity and boundedness assumptions
on the coefficients defining the system. In particular, Law of Large
Numbers, Central Limit Theorems and Large Deviation Principles have
been established (see for example
\cite{cf:szni,cf:mele,cf:tanaka,cf:DG,cf:BDF}). A number of important
issues remain unsolved, like the generalization to singular
interactions (e.g. \cite{cf:jabin}) or understanding the delicate
issue of considering at the same time large $n$ and large time
(e.g. \cite{cf:LP}). But another direction in which mathematical
results are still very limited is about relaxing the \emph{complete
  graph assumption} for the interaction network -- \emph{complete
  graph} is just a different wording for \emph{mean field} -- and
going towards more heterogeneous interaction networks.  This is an
issue that emerges in plenty of applied disciplines and giving a
proper account of the available literature would be a daunting task:
so we limit ourselves to signaling the recent survey \cite{cf:survey}
which contains an extended literature.  

%\smallskip

We are therefore going to study the emerging behavior of
  interacting diffusion models when, like in complete graphs, every
  unit interacts with a diverging number of other units. The
  interaction network is described as a random graph, notably of Erd\H
  os-R\'enyi (ER) type; so we start with the basic notions on 
  graphs.

Let $\xi^{(n)}=\{ \xi_{i,j}^{(n)}\} _{i,j\in\{1,\dots,n\}}$ denote the adjacency matrix of a graph $\left(V^{(n)},E^{(n)}\right)$ with $n$ vertices ($\xi^{(n)}$ will also denote the graph itself):
\begin{equation}
V^{(n)}\, :=\, \left\{ 1,\dots,n\right\} \ \text{ and } \
E^{(n)}\,:=\, \left\{ (i,j)\in V^{(n)}\times V^{(n)}:\xi_{i,j}^{(n)}=1\right\} \,.
\end{equation}
We consider sequences of asymmetric ER random
  graphs with self loops with probabilities \(p_n \in (0,1)\) for
  \(n=2,3,\dots\). More precisely, we just assume that
  $\{ \xi_{i,j}^{(n)}\} _{i,j\in\{1,\dots,n\}}$ are Independent
  Identically Distributed (IID) Bernoulli random variables of
  parameter $p_n$ (with notation B($p_n$)). The arguments are easily
  adapted to the case in which $\xi_{j,j}^{(n)}=0$ for every $j$ and
  the results are unchanged. 

\smallskip

Even if these graphs are not coupled for different values of $n$, it is practical to work with
only one probability space and to couple  these  adjacency matrices (or random graphs). For example one can start from a sequence $\{U_k\}_{k \in \bbN}$ of IID $U(0,1)$ variables and define $\xi^{(n)}_{i,j}= \ind_{U_{k(i,j)}<p_n}$, with $k$ an arbitrary bijection from $\bbN^2$ to $ \bbN$. The law of the graph is denoted by $\bbP$, with $\bbE$ the corresponding expectation,  and we will just write $\bbP(\dd \xi)$-a.s. meaning ``almost surely in the realization of $\{\xi^{(n)}\}_{n=2,3, \ldots}$".

\medskip

Given a realization of $\xi^{(n)}$, consider the $n$-dimensional diffusion $\theta_{t}^{n}:=\{\theta_{t}^{i,n}\}_{i=1, \ldots, n}$ which solves for every $i$
\begin{equation}
\label{eq:dRG}
\dd\theta_{t}^{i,n}\, =\, F\left(\theta_{t}^{i,n}\right)\dd t+\frac{1}{n}\sum_{j=1}^{n}\frac{\xi_{i,j}^{(n)}}{p_{n}}\Gamma\left(\theta_{t}^{i,n},\theta_{t}^{j,n}\right)\dd t+\sigma \left(\theta_{t}^{i,n}\right) \dd B_{t}^{i},
\end{equation}
where  $\{B_{\cdot}^{i}\}_{i\in \bbN}$ are independent standard Brownian motions (whose law is denoted by $\bP$) and  independent also of $\xi^{(n)}$ (so, we are effectively working with $\bP \otimes \bbP$). For simplicity, we consider only deterministic initial conditions; but the results apply to random initial conditions once they are taken independent of Brownian motions and of $\xi$. Moreover, assume that:
\smallskip

\begin{enumerate}
\item
$F$, $\Gamma$ and $\gs$ are real valued (uniformly) Lipschitz functions: the corresponding Lipschitz constants are denoted by $L_F$, $L_\Gamma$ and
$L_\gs$;
\item $\Gamma$ is bounded, in particular  $\norm{\Gamma}_\infty:=\sup_{x,y\in\R}\left|\Gamma(x,y)\right| < \infty$;
\item  $\gs_- \le \gs (\cdot) \le \gs_+$ with $\gs_\pm$ two positive  constants (non degenerate diffusion). If   $\gs (\cdot)$ is a constant, we include the case $\gs (\cdot)\equiv 0$.
\end{enumerate}

\smallskip

Fix $T>0$, the law of the $n$ trajectories $\{\theta_{t}^{n}\}_{t \in [0,T]}$ for the \textit{quenched} system is denoted by $\bP^\xi_n$, i.e. $\bP^\xi_n \in \cP \left(\cC^0 ([0,T];\R^n)\right)$, and the associated empirical measure at time $t$ by $\left\{\mu_t^n\right\}_{t\in[0,T]}$, i.e.
\begin{equation}
\label{eq:empMesT}
\mu_t^n \, :\, = \frac 1n \sum_{j=1}^n \delta_{\theta_{t}^{j,n}} \in \cP \left( \bbR \right)\, .
\end{equation}
$\cP \left( \bbR \right)$ denotes the set of probability measures over $(\bbR, \cB(\bbR))$ equipped with the (metrizable) topology of weak convergence: i.e., if $\mu_n\in\cP \left( \bbR \right)$ for every $n$, then $\lim_n \mu_n = \mu \in \cP \left( \bbR \right)$ if $\int h(x) \mu_n(\dd x) \to  \int h(x) \mu(\dd x)$  as $n\uparrow \infty$ for every $h(\cdot)$ continuous and bounded function.
Note that since $\xi$ is random, $\mu_t^n$ is a random variable taking values in $\cP \left( \bbR \right)$, equipped with the $\gs$-algebra of its Borel subsets.

The solution  $\{\theta_{t}^{i,n}\}_{i=1, \ldots, n}$ is going to be tightly linked with $\{\bar \theta_{t}^{i,n}\}_{i=1, \ldots, n}$ which solves
\begin{equation}
\label{eq:dnonRG}
\dd{\bar{\theta}}_{t}^{i,n}\, =\, F\left(\bar \theta_{t}^{i,n}\right)\dd t+\frac{1}{n}\sum_{j=1}^{n}\Gamma\left(\bar\theta_{t}^{i,n},\bar\theta_{t}^{j,n}\right)\dd t+\sigma \left(\bar\theta_{t}^{i,n}\right) \dd B_{t}^{i}.
\end{equation}
The law of $\{\bar \theta_{t}^{n}\}_{t \in [0,T]}$ is denoted by $\bP_n$. Moreover
$\bar \mu_t^n \, :\, = (1/n) \sum_{j=1}^n \delta_{\bar\theta_{t}^{j,n}}$.
Often  (\ref{eq:dnonRG}) is called  \textit{annealed} system: of course  (\ref{eq:dnonRG}) is obtained from
 (\ref{eq:dRG}) by taking the expectation of the drift with respect to $\bbP$.

\medskip

If the empirical measure of the initial conditions converges to a  probability $\mu_0$, i.e.
\begin{equation}
\label{eq:hypmu0}
\lim_{n\rightarrow \infty} \bar \mu_0^n \, =\,  \mu_0 \in \cP \left( \bbR  \right)
\,,
\end{equation}
and if  $\int_\bbR x^2 \mu_0(\dd x) < \infty$,
then it is well known that, for every $t >0$, $\bar \mu_t^n$ weakly converges in $\cP \left( \bbR  \right)$ to $\mu_t$, the unique weak solution of the following McKean-Vlasov (or Fokker-Planck) equation
\begin{equation}
\label{eq:McKV}
\partial_t \mu_t (\theta) \, =\,  \frac  12 \partial^2_{\theta }\left({\sigma^2(\theta)} \mu_t (\theta)\right)-\partial_{\theta} \left( \mu_t (\theta) F(\theta) \right) - \partial_\theta \left( \mu_t (\theta) \int_{\bbR} \Gamma(\theta, \theta') \mu_t (\dd \theta') \right)\, .
\end{equation}
The slightly stronger result that is proven is in fact: for every $T>0$, if one considers $\mu^n_\cdot$ as an element of $C^0([0,T]; \cP(\bbR))$ (a complete separable metric space), then $\lim_n \mu^n_\cdot = \mu_\cdot$ ($\bP$-a.s. when $\sigma$ is non degenerate). The notion of weak solution $\mu_\cdot \in C^0([0,T]; \cP(\bbR))$ to \eqref{eq:McKV}, which can be found for example in \cite{cf:G88}, is strictly related to the nonlinear diffusion formulation: the stochastic process $\{\gp _t\}_{t\in [0, T]}$ that solves
\begin{equation}
\label{eq:theta_nonlin}
\begin{cases}
\dd\gp_{t}\,=\,  F\left(\gp _{ t}\, \right)\dd t +  \int \Gamma\left( \gp_{ t},  \gp  \right) \nu_t \left(\dd \gp \right) \dd t + \sigma \left(\gp_t \right)\dd B_{t}\, , \\
\nu_t \,=\,  {\mathrm{Law}} (\gp_t)\,, \quad\text{for all }t\in[0,T]\, ,
\end{cases}
\end{equation}
with initial condition which is a square integrable random variable independent of the standard Brownian motion $B_\cdot$. Existence and uniqueness for this \emph{atypical} stochastic differential equation is not obvious at all, but it is by now well known that if $\nu_0 = \mu_0$, then the unique $\nu_\cdot \in C^0([0,T]; \cP(\bbR))$ such that $\nu_t$ is the law of $\gp _t$ for all $t\in [0,T]$, is the unique weak solution of \eqref{eq:McKV}, i.e. $\nu_t =\mu_t$ for all $t\in [0,T]$.  The literature on the results that we have just mentioned is vast, see e.g.   \cite{cf:oel,cf:szni,cf:mele,cf:G88} for the non degenerate diffusion case and \cite{cf:dobrushin,cf:neunzert} for the $\gs(\cdot)\equiv 0$ case; in this last case there is no need to assume that $\int_\bbR x^2 \mu_0(\dd x) < \infty$.

In the sequel, we will also work with probabilities in $\cP \left(\cC^0 ([0,T]; \R)\right)$, that is considering the law of $\{\gp_{t}\}_{t\in[0,T]}$ seen as a random trajectory on the path space $\cC^0 ([0,T]; \R)$, rather than its time marginals $\mu_t \in \cP(\R)$.

\medskip

\begin{rem}
\label{rem:proj}
	Observe that knowing the law of \eqref{eq:theta_nonlin} gives more information than the solution $\mu_\cdot$ of the McKean-Vlasov equation \eqref{eq:McKV}. Indeed, call $P_\gp$ the law  of $\{\gp _t\}_{t\in [0, T]}$, then $P_\gp$ is an element of $\cP \left(\cC([0,T]; \R) \right)$, whereas $\mu_\cdot \in C^0([0,T]; \cP(\bbR))$. It is straighforward to obtain $\mu_t$ from $P_\gp$ by just observing
	\begin{equation}
	\label{eq:mapPI}
	\mu_t (\cdot) = P_\gp \circ \pi_t^{-1} (\cdot),
	\end{equation}
	where $\pi_t\,:\,\cC([0,T]; \R) \rightarrow \R$ is the canonical projection at time $t$. Observe that a reverse statement is not always possible: $\mu_\cdot$ alone does not allow to compute multidimensional time marginals like $\bP (\gp_s \in A, \gp_t \in B)$, for $s,t \in [0,T]$ and $A,B \subset \R$. Existence, uniqueness and well-posedness of the problem for $P_\gp$ can be found in \cite{cf:mele} and references therein.
\end{rem}

%{\blue If $\sigma \equiv 0$, a duality between the (degenerate) non linear process and a partial differential equation (this time called \emph{Vlasov equation}) still holds under stronger conditions on the initial datum, see \cite{cf:BH77} for more details. GB: I put back Dobrushin and Neunzert: what are the extra conditions? Dobrushin and Neunzert do it for probability initial data. }

\subsection{Aim of the paper}
Informally stated, our aim is to study the proximity of $\mu^n_\cdot$ and $\bar \mu^n_\cdot$, for $n$ large. Since $\bar \mu^n_\cdot$ approaches the solution of the McKean-Vlasov equation \eqref{eq:McKV}, this turns out to be studying  the proximity of $\mu^n_\cdot$ and the solution of the McKean-Vlasov equation. This of course requires (at least) the assumption that
\begin{equation}
\label{eq:hypmu0.0}
\lim_{n\to \infty} \mu^n_0\, = \,  \mu_0\, .
\end{equation}

\medskip

A result of this type has been already achieved: in the case  $\gs (\cdot) \equiv \gs \ge 0$,  \cite{cf:DGL} proved a LLN for the trajectories of \eqref{eq:dRG} where $\xi^{(n)}$  is a (deterministic) sequence of graphs such that
\begin{equation}
	\label{hyp:DGL}
	\lim_{n \to \infty} \sup_{i\in\{1,\dots,n\}} \left|\frac 1n \sum_{j=1}^n \frac{\xi^{(n)}_{i,j}}{p_n} -1 \right|\,=\,0\,,
\end{equation}
and with IID initial conditions (\emph{chaotic} initial datum), that is $\theta^{j,n}_0=\theta^{j}_0$ for every $n$ and every $j=1, \ldots, n$ where
\begin{equation}
\label{hyp:DGL0}
\left\{ \theta^{j}_0\right\}_{j\in \bbN} \, \text{ is a typical  realization of an IID sequence of variables with law } \mu_0\, .
\end{equation}
Under conditions \eqref{hyp:DGL} and \eqref{hyp:DGL0}, it is proved that $\lim_n \mu^n_\cdot = \mu_\cdot$ in $\bP$-probability. We recall that, as stated right after  \eqref{eq:dRG},
$\{ \theta^{j}_0\}_{j\in \bbN}$ is independent of the driving Brownians and of the graph $\xi$.

This seems at first rather satisfactory. However in \cite{cf:DGL} it is discussed at length how  this result in reality is, on one hand,  surprising and, on the other, that it does not really solve the problem. This can be understood by considering that
the \emph{homogeneous degree condition} \eqref{hyp:DGL} is $\bbP \left(\dd \xi \right)$-a.s. verified for  ER type graphs when $\liminf_n  n p_n/ \log n$ is larger than a well-chosen constant (see \cite[Proposition 1.3]{cf:DGL}). But the class of graphs satisfying \eqref{hyp:DGL}  goes well beyond ER graphs: in particular, it is straightforward to construct graphs with an arbitrary number of connected components that satisfy \eqref{hyp:DGL}, see the following remark.

%\item The Law of Large Numbers (LLN) readily implies that, if we assume chaotic initial conditions (i.e., \eqref{hyp:DGL0}), then $\bbP\left(\dd \xi \right)$-a.s. \eqref{eq:hypmu0.0} holds.  But of course  \eqref{hyp:DGL0} is a much more stringent condition than \eqref{eq:hypmu0.0}.
%\end{enumerate}

\smallskip
\begin{rem}
	\label{rem:eo}
	\eqref{eq:hypmu0.0} and \eqref{hyp:DGL} are not sufficient to obtain a result in the direction
	we are aiming at. In fact, if $\xi^{(n)}$ is the graph in which two vertices are connected if and only if they have the same parity (which corresponds to $\lim_n p_n=1/2$), then, as long as $\mu_0$ is not the uniform measure, one can easily arrange the initial condition in order to have different limit distributions on even and odd sites, or no limit at all. Thus, as $n \to \infty$, the evolution will not be described by \eqref{eq:McKV}.
\end{rem}

In a nutshell, the results in \cite{cf:DGL} are obtained under a weak assumption on the graph, but under strong assumptions on the initial condition. And this to the point of obtaining a result that is troublesome: a system with plenty of disconnected components behaves essentially like a totally connected one! Of course the \emph{solution} of this apparent paradox is in the chaotic character of the initial condition that leads to a homogeneous and identical behavior of the initial datum on all components, and the fact that chaos propagates at least on a finite time horizon (see \cite{cf:DGL} for more on this issue).
But there is no reason to expect mean field type behavior, assuming only \eqref{hyp:DGL} on the graph, without a  strong \emph{statistical homogeneity}  assumption on the initial datum, as argued in Remark~\ref{rem:eo}.

The aim of this paper is to attack the problem assuming only the convergence of the empirical measure of the initial datum, that is \eqref{eq:hypmu0.0}, but assuming that the graph is of ER type. Otherwise said, we want to make a minimal assumption on the initial condition and we try to exploity the \emph{chaoticity} of the graph to achieve the result. We will attack the problem from more then one perspective, not only the direct LLN angle of attack, but also from the Large Deviations (LD) perspective. The vast literature related to our results
is presented and discussed after the statements.

\section{Main results}

Let us denote $d_{\text{bL}}(\cdot, \cdot)$ the bounded Lipschitz distance which endows the weak convergence topology on $\cP(\bbR)$ (this choice is somewhat arbitrary: other distances can be used, for example the Wasserstein one, see \cite{cf:dobrushin}). By this we mean that $d_{\text{bL}}(\mu, \nu)= \sup_{h}  \left|\int h \dd \mu - \int h \dd \nu\right|$, where the supremum is taken over $h: \bbR \to [0,1]$ such that $\vert h(x)-h(y)\vert \le \vert x-y \vert$.

\medskip
We are now ready to state the LLN. Recall that $\mu^n_\cdot$ is a random element of $C^0([0,T]; \cP(\bbR))$
and that $\mu_\cdot$, a non random element of $C^0([0,T]; \cP(\R))$,   is  the unique weak solution of the McKean-Vlasov equation \eqref{eq:McKV}.
\medskip

\begin{theorem}
\label{th:LLN0}
Assume that the initial datum is deterministic, that it satisfies  \eqref{eq:hypmu0.0} and, if $\gs (\cdot) \not \equiv 0$, that it satisfies also    that $\int_\bbR x^2 \mu_0(\dd x) < \infty$. Make the hypothesis that $p_n$ satisfies
\begin{equation}
\label{eq:p_n-1}
\liminf_{n \to \infty}  \frac{p_n n}{\log n} \,  >0\, ,
\end{equation}
and either that $0<\gs_- \le \gs (\cdot) \le \gs_+< \infty$ or  $\gs (\cdot)\equiv 0$. Then  $\bP\otimes\bbP$-a.s.  we have that
\begin{equation}
\label{eq:LLN-ER}
\lim_{n \to \infty}\mu^n_\cdot = \mu_\cdot \quad \text{ in }C^0([0,T]; \cP(\R)).
\end{equation}
%Then  $\bbP(\dd \xi)$-a.s.  we have that $\lim_{n \to \infty}\mu^n_\cdot = \mu_\cdot$  in $\bP$-probability.
% More explicitly: for every $\gep>0$
%\begin{equation}
%\label{eq:LLN0-expl}
%\lim_{n \to \infty} \bP \left( \sup_{t \in [0, T]} d_{\textnormal{bL}}\left( \mu^n_t, \mu_t \right)\le \gep \right)\, =\, 1\, , \quad  \bbP(\dd \xi)\text{-a.s.}.
%\end{equation}
\end{theorem}

\medskip

The requirement of deterministic initial data is easily lifted to IID initial conditions under the assumption
that they are independent of the graph (and, of course, of the driving Brownians).
%: treating the deterministic case directly yields that the result
%holds for arbitrary initial laws satisfying \eqref{eq:hypmu0.0} in $\bP$-a.s. and, in the diffusion case,   $\int x^2\mu_0(\dd x)
%<\infty$

\medskip

From the viewpoint of the proof, Theorem~\ref{th:LLN0} may be viewed as two different statements.

\smallskip

\begin{itemize}
\item in the case of $\gs (\cdot)\equiv \gs \in [0, \infty)$, the proof follows by coupling the system on the ER graph and the system on the complete graph;
\item in the case of $0<\gs_- \le \gs (\cdot) \le \gs_+< \infty$, the result is a corollary of  a Large Deviation Principle (LDP) stating that, at the Large Deviations (LD) level, the system on ER graph and the complete graph system are indistinguishable, see Theorem \ref{thm: LDP}.
\end{itemize}

\medskip
In the next subsection we present the result related to Large Deviations.

\medskip

\subsection{The Large Deviation Principle}
Stating the LDP needs some preparation on the general LD approach (classical references are for example \cite{cf:DZ,cf:dupuis,cf:dH}).

Given a complete, separable metric space $\chi$, a rate function $I$ is a lower semicontinuous mapping $I:\chi \rightarrow [0,\infty]$ such that each level set $K_l= \left\{ x\in \chi : I(x) \le l \right\}$ is compact for all $l\ge0$ (sometimes $I$ is called a \textit{good} rate function).
Given $\left\{ P_n \right\}_{n\in\bbN}$ a sequence of probability measures on $\chi$ associated with its Borel $\sigma$-field, we say that $P_n$ satisfies a LDP (on $\chi$) with  rate function $I$ if for every measurable set $A\subset \chi$
\begin{equation}
	\label{eq:LDP}
-\inf_{x\in A^\circ}I(x) \, \le \, 	\liminf_{n\rightarrow\infty}\tfrac{1}{n}\log P_{n}\left(A^\circ\right)\,
\le \,
\limsup_{n\rightarrow\infty}\tfrac{1}{n}\log P_{n}\left(\bar{A}\right)\le-\inf_{x\in \bar{A}}I(x) ,
\end{equation}
where $A^\circ$ is the interior of $A$ and $ \bar{A}$ is its closure.

\medskip
Let us now recall that \eqref{eq:dnonRG}, or equivalently \eqref{eq:dRG} on a complete graph, satisfies a LDP, we refer to \cite[Theorem 3.1]{cf:BDF}. We choose to state the LDP for the empirical law of the process, that is for
\begin{equation}
\label{eq:L}
\overline{L}_n\, :=\,
\frac 1n \sum_{j=1}^n \gd_{\bar \theta_\cdot ^{j,n}} \in \cP \left( C^0([0,T]; \bbR) \right),
\end{equation}
but other LDP are possible. Namely, \cite[Theorem 5.1]{cf:DG} proves a LDP for the empirical measure $\bar{\mu}_\cdot^n$ seen as an element of $C^0\left([0,T];\cP(\R)\right)$ (recall \eqref{eq:mapPI}), yet our result includes this case. In Remark~\ref{rem:proj} we have pointed out the continuity of the projection $\pi_t$ and how to pass from $P_n$ to $\mu^n_\cdot \in C^0([0, T]; \cP(\bbR))$, therefore a corollary of a LDP for $\bar{L}_n$ is a LDP on $C^0([0, T]; \cP(\bbR))$ for the law of $\bar{\mu}^n_\cdot$ with LD functional given by the contraction principle: see for example \cite{cf:DPdH,cf:lucon} for the mathematical procedure and  \cite{cf:DG} for an explicit form of the LD functional in the full generality.

\medskip

We set  $\chi=  \cP \left( C^0([0,T]; \bbR) \right)$; since $C^0([0,T]; \bbR)$ is a metric space, $\chi$ is a complete, separable metric space once equipped (among various possibilities) with the  bounded Lipschitz distance. Define the probability measure $\overline P_n$ on $\chi$ by setting $\overline P_n(\cdot):= \bP( \overline{L}_n \in \cdot)$, of course $\chi$ equipped with the $\gs$-algebra of its Borel subsets, then \cite[Theorem 3.1]{cf:BDF} shows that $\overline P_n$ satisfies a LDP whose rate function concentrates on $\nu \in \chi$ such that $\nu_\cdot = \nu \circ \pi^{-1}_\cdot \in C^0([0,T]; \cP(\R))$ is solution of the McKean-Vlasov equation \eqref{eq:McKV}.
%We give an explicit form of the rate function in the so-called Hamiltonian case at the end of the section.

\medskip

We are now ready to state the main result of this subsection. For every realization of the graph $\xi$ define  the probability $P^{\xi}_{n}$  on $\chi$, by setting $P^{\xi}_{n}\left( \cdot \right) := \bP \left( L^n \in \cdot \right)$, where $L^n (\cdot)$ is defined as in \eqref{eq:L}, but replacing $\bar \theta_\cdot ^{j,n}$ with $ \theta_\cdot ^{j,n}$. In particular, $P^{\xi}_{n}$ is the empirical measure of the trajectories $ \theta_\cdot ^{j,n}$ solving \eqref{eq:dRG}.

\medskip

\begin{theorem}
	\label{thm: LDP}
	Assume that $\gs_->0$. If $\xi$ is an ER graph that satisfies \eqref{eq:p_n-1} and if the initial datum satisfies \eqref{eq:hypmu0.0} and $\int x^2\mu_0(\dd x)
	<\infty$, then $P^{\xi}_{n}$ satisfies the same LDP of $P_n$ $\bbP(\dd \xi)$-a.s..
\end{theorem}

\medskip

%\subsubsection{The gradient case}
%$\dots$
%
%If $Q\in \chi$, for every $t\in [0,T]$ we can define $Q_t \in \cP(\R)$  by setting $Q_t(B):= Q(\{\{x_s\}_{s\in[0,T]}:\, x_t\in B \})$ for every Borel subset of $\bbR$. It is straightforward to see that $Q_\cdot \in C^0([0,T]; \cP(\bbR))$ and we can then consider the unique strong solution of the stochastic ODE
%\begin{equation}
%\label{eq:phi2}
%	\qquad \dd\phi_{t}\,=\,  F\left(\phi _{ t}\, \right)\dd t +  \int \Gamma\left( \phi_{ t},  \phi  \right) Q_{ t}\left(\dd \phi \right) \dd t + \sigma \left(\phi_t \right)\dd B_{t}\, ,
%\end{equation}
%with the law of $\phi_0$ given by $Q_0$. Call $P^Q$ the law of $\{\phi_s\}_{s \in [0,T]}$. Then $I(Q)$ is just the relative entropy of $P^Q$ with respect to $Q$, namely, $I(Q)= \int_\chi \log (\dd P^Q/ \dd Q) \dd P^Q$ if  $P^Q$ is absolutely continuous with respect to $Q$, and $I(Q)=\infty$ otherwise.

\subsection{A look at the literature}
We recall that for interacting particle systems on the complete graph, i.e. \eqref{eq:dnonRG}, many results on the LLN are available and many of them, as \cite{cf:DGL}, include propagation of chaos properties. However, as already mentioned, propagation of chaos results are very demanding on the initial condition.

The literature is vast and difficult to be properly cited:  we mention the seminal contribution   \cite{cf:mckean66}   and  we mention again \cite{cf:oel,cf:szni,cf:mele,cf:G88}, that are also useful source of more references and that are not limited to propagation of chaos results, in the sense that also the case of deterministic initial data is treated.
For the $\gs(\cdot)\equiv 0$ case, we mention the important original contributions \cite{cf:dobrushin,cf:neunzert}
that gave origin to  a vast literature that goes beyond our purposes.

Large deviation properties for mean field diffusions have been studied in the seminal work by Dawson and G\"{a}rtner \cite{cf:DG}, but also in \cite{cf:dH,cf:fengKurtz,cf:lucon} in the so called gradient case. In \cite{cf:BDF} the problem is attacked in great generality using an approach based on weak convergence and control theory.

The LLN case has already been adressed in the literature, even if few results seem to have been proven so far. As mentioned, \cite{cf:DGL} proves a LLN for $\mu^n_\cdot$ requiring the initial datum to be a product measure: the case $\gs(\cdot)\equiv \gs \ge 0$ is considered.  In the same spirit, from the initial datum viewpoint, but for a time-varying graph and for multi-type processes, there is the work of \cite{cf:BBW}. It is important to mention at this stage that in \cite{cf:BBW} the interaction is renormalized by the number of neighbors of each site $i$: we normalize instead by the expected number of neighbors.

Turning to LD results, the recent work of \cite{cf:reisOliv} extends the LDP for Hamiltonian systems in random media, presented in \cite{cf:DPdH}, to (sparse) random interactions which include symmetric ER random graphs. The convergence of the empirical measure is shown under the assumption $\lim_n np_n = \infty$, without requiring any $\log$ divergence. However, they still focus on IID initial conditions and constant diffusion term $\sigma (\cdot) \equiv 1$.

\smallskip
Focusing on the case $\gs(\cdot)\equiv 0$, we mention the contributions:
\smallskip

\begin{itemize}
\item in
\cite{cf:brecht} one finds the stability analysis  for  the stationary state of an ordinary differential equations  system with ER interacting network, requiring a
logarithmic divergence of $p_n n$;
\item in \cite{cf:chibaMed}
 the Kuramoto model, i.e.  $\Gamma (x,y) = \sin (x-y)$ and $F(\cdot)$ is a random constant (\emph{natural frequencies}), is studied with an interaction network that is given by a \emph{graphon}: this leads to a more general limit equation, but their approach includes the case of ER graphs (in this case the graphon is trivial) with
 $p_n$ that tends to a positive constant.
 In  \cite{cf:medvedev}  the case of sparse graphs is considered: for ER graphs the condition is
 $\lim_{n}p_n \sqrt{n}=\infty$.
 \end{itemize}

\medskip

In many of the papers we cite, notably \cite{cf:DGL,cf:DPdH,cf:lucon,cf:chibaMed,cf:medvedev}, another source
of randomness is allowed: for example, in the Kuramoto model this corresponds to the important feature that each oscillator has a priori its own oscillation frequency and, more generally, with this extra source of randomness we can model systems in which
the interacting diffusions (or units, agents,$\ldots$) are not identical. This source of randomness is chosen independently of the graph and of the dynamical noise. All the results we have presented generalize easily to this case, but at the expense of heavier notations and heavier expressions. We have chosen not to treat this case for sake of conciseness and readability.

\medskip

The rest of the paper is devoted to the proofs. Section~\ref{sec:LLN} contains  the proof of Theorem~\ref{th:LLN0} in the case of constant (possibly degenerate) diffusion coefficient.  Section~\ref{sec:LDP} contains
the proof of  Theorem~\ref{thm: LDP}.

\section{The Law of Large Numbers: the proof}
\label{sec:LLN}
This section is devoted to the proof of Theorem~\ref{th:LLN0} in the case $\gs(\cdot)\equiv \gs \in [0, \infty)$: we recall that the case of non trivial and non degenerate diffusion is a corollary of the LDP (Theorem~\ref{thm: LDP}).
We start with  two preliminary lemmas that will be used for Proposition \ref{th:wips}, from which 
Theorem~\ref{th:LLN0} follows.

\medskip
\begin{lemma}
	\label{lem:max-degree}
	Let $K>2$. For all $n\in \N$, it holds
	\begin{equation}
	\label{eq:max-degree}
	\p \left( \sum_{j=1}^{n} \left|\frac{\xi_{i,j}^{(n)}}{p_n} - 1 \right| \ge K n \right) \le \exp
	\left(-\frac{3(K-2)^2 }{6 + 2(K-2)} p_n n\right).
	\end{equation}
	In particular, under hypothesis \eqref{eq:p_n-1} and setting $C:=\liminf_{n \to \infty}  \frac{p_n n}{\log n}\in (0, \infty]$, we have that,  if $K>K_C := 2 + \frac 2{3C} +  \sqrt{\frac{4}{9C^2} + \frac{4}{C}}$,  then    $\bbP(\dd \xi)$-a.s.
	there exists $n_0=n_0(\xi)< \infty$ such that for $n \ge n_0$
	\begin{equation}
	\label{eq:max-degree2}
	\max \left(\sup_{i=1,\dots,n} \sum_{j=1}^{n} \left|\frac{\xi_{j,i}^{(n)}}{p_n} - 1 \right|, \sup_{i=1,\dots,n} \sum_{j=1}^{n} \left|\frac{\xi_{i,j}^{(n)}}{p_n} - 1 \right| \right)  \le Kn\, .
	\end{equation}
\end{lemma}

\medskip

\begin{proof}
	We  use Bernstein's inequality (see for example \cite[Corollary 2.11]{cf:bouche}) which says that if $X_1, \dots, X_n$ are independent zero-mean random variables such that $|X_j| \le M$ a.s. for all $j$, then for all $t\ge 0$
\begin{equation*}
\p \left( \sum_{j=1}^{n} X_j > t \right) \,\le\,  \exp\left\{-\frac{\frac 12 t^2}{\sum_{j=1}^n \bbE[X_j^2] + \frac 13 Mt} \right\}.
\end{equation*}
Set $X_j = \left|\frac{\xi_{i,j}^{(n)}}{p_n} - 1 \right| - 2(1 - p_n)$. $X_j$ is a zero-mean random variable and we can bound
\begin{equation}
\p \left( \sum_{j=1}^{n} \left|\frac{\xi_{i,j}^{(n)}}{p_n} - 1 \right| \ge Kn \right) \le \p \left(  \sum_{j=1}^{n} X_j \ge (K-2)n \right).
\end{equation}
We have $\vert X_j\vert \le \max(1/p_n -3+2p_n, 2p_n-1) \le 1/p_n=:M$ and
$\bbE[X_j^2] -1/p_n= -5 + 8 p_n - 4 p_n^2 \le -1$, so $\bbE[X_j^2] \le 1/p_n$
and  Bernstein's inequality together with an union bound show that
\begin{equation}
\p \left( \sup_{i=1,\dots,n} \sum_{j=1}^{n} \left|\frac{\xi_{i,j}^{(n)}}{p_n} - 1 \right| \ge K n \right) \le n \exp
\left(-\frac{3(K-2)^2 }{6 + 2(K-2)} p_n n\right)\, .
\end{equation}
%Under \eqref{eq:p_n-1}, the statement holds in probability if, for $K=K_C$,
%\begin{equation}\frac{3(K_C-2)^2 }{6 + 2(K_C-2)} > \frac 1C.\end{equation}
%Therefore, it suffices to choose $K_C$ such that
%\begin{equation}
%K_C > 2 + \frac 1{3C} +  \sqrt{\frac{1}{9C^2} + \frac{2}{C}}.
%\end{equation}
%If we want the a.s. result, a straightforward application of Borel-Cantelli gives the rest of the proof.
The proof is now completed with some elementary computations and by applying the Borel-Cantelli Lemma.
\end{proof}

\medskip
\begin{lemma}
	\label{lem:exp}
	Assume \eqref{eq:p_n-1} and	let
	\begin{equation}
	\label{eq:delta_i}
	\Delta_i (s) :=  \left| \frac{1}{n} \sum_{j=1}^{n} 	\left(\frac{\xi_{i,j}^{(n)}}{p_{n}} - 1\right) \Gamma\left(\bar\theta_s^{i,n},\bar\theta_s^{j,n}\right) \right|^2, \quad \text{for every } s \in [0,T].
	\end{equation}
	Then, for every realization of the Brownian motions, it holds that
	\begin{equation}
	 \lim_{n \to \infty} \int_0^T \frac{1}{n} \sum_{i=1}^{n} \Delta_i (s) \dd s = 0, \quad  \bbP(\dd \xi)\text{-a.s.}.
	\end{equation}
\end{lemma}

\begin{proof}
	First, we rewrite $\int_{0}^{T} \Delta_i (s) \dd s$ as
\begin{equation}
\int_{0}^{T} \Delta_i (s) \dd s = \frac{1}{(np_n)^2}\sum_{j,k=1}^n \hat{\xi}_{i,j} \hat{\xi}_{i,k} d_{ijk},
\end{equation}
where we have dropped the superscript $(n)$, the dependency on $T$ and we have introduced the notations
\begin{equation}
\label{eq:notations}
 \hat{\xi}_{i,j}\, :=\,  \xi_{i,j}^{(n)} - p_n\ \ \ \text{ and } \ \ \
d_{ijk} \,:=\,  \int_{0}^{T}  \left[\Gamma\left(\bar\theta_s^{i,n},\bar\theta_s^{j,n}\right) \Gamma\left(\bar\theta_s^{i,n},\bar\theta_s^{k,n}\right)\right] \dd s\, .
\end{equation}
Observe that $\hat{\xi}_{i,j}$ are centered random variables and $\left|d_{ijk} \right| \le T \norm{\Gamma}_\infty^2 =: d_\star$.

Let $\delta_n$ be a sequence of positive numbers such that (recall \eqref{eq:p_n-1})
\begin{equation}
	\label{eq:constant}
	\delta_n \gg \frac 1{p_nn} \ \ \text{and}  \ \ \lim_{n\to \infty}\delta_n = 0.
	\end{equation}
Let $\Omega_n$ be the set
\begin{equation}
\label{eq:OmeganLLN}
	\Omega_n \, :=\, \left\{ \xi : \frac{1}{(np_n)^2}\sum_{i,j,k=1}^n \hat{\xi}_{i,j} \hat{\xi}_{i,k} d_{ijk}  > \delta_n n  \right\}\, ,
\end{equation}
We want to show that
\begin{equation}
\sum_{n \in \bbN} \bbP \left( \Omega_n \right) < \infty.
\end{equation}
Let $K>2$ and consider the events
\begin{equation}
\label{eq:An}
A_n = \bigcup_{i=1}^{n} A_{n,i} \ \  \text{ with } \ \  A_{n,i} = \left\{ \xi^{(n)} : \sum_{j} \left|\frac{\xi^{(n)}_{i,j}}{p_n} - 1\right| > Kn.
\right\}.
\end{equation}
We use
\begin{equation}
\bbP \left( \Omega_n \right) \le \bbP \left( \Omega_n \cap A_n^\complement \right) + \bbP \left( A_n \right).
\end{equation}
and Lemma \ref{lem:max-degree} ensures  that choosing  $K>K_{C}(>2)$ we have
\begin{equation}
\label{eq:An-con}
\sum_{n\in\N} \bbP(A_n) < \infty,
\end{equation}
so that one is left with proving that $\sum_{n \in \bbN} \bbP \left( \Omega_n \cap A_n^\complement \right) < \infty$.
By Markov's inequality applied to $\bbP\left(\, \cdot \, \vert  A_n^\complement \right)$ we see that
\begin{equation}
\bbP(\Omega_n \cap A_n^\complement)
\le \exp\left(-n \delta_n + \log \bbE \left[\ind_{A_n^\complement}\exp\left(\frac{1}{(np_n)^2}\sum_{i,j,k=1}^n \hat{\xi}_{i,j} \hat{\xi}_{i,k} d_{ijk}\right)\right] \right).
\end{equation}
Given \eqref{eq:constant}, it suffices to show that
\begin{equation}
\label{goalLLN}
\log \bbE \left[\ind_{A_n^\complement}\exp\left(\frac{1}{(np_n)^2}\sum_{i,j,k=1}^n \hat{\xi}_{i,j} \hat{\xi}_{i,k} d_{ijk}\right)\right] \,=\, n\; O\left(\frac 1{p_nn}\right)\, =\, O\left(\frac 1{p_n}\right) \, .
\end{equation}

We exploit the independence w.r.t. $i$:
\begin{equation}
\label{eq:2fLLN}
\bbE \left[\ind_{A_n^\complement}\exp\left(\frac{1}{(np_n)^2}\sum_{i,j,k=1}^n \hat{\xi}_{i,j} \hat{\xi}_{i,k} d_{ijk}\right)\right] = \prod_{i} \bbE \left[\ind_{A_{n,i}^\complement}\exp\left(\frac{1}{(np_n)^2}\sum_{j,k=1}^n \hat{\xi}_{i,j} \hat{\xi}_{i,k} d_{ijk}\right)\right]\, ,
\end{equation}
and use the inequality $\exp(x) \le 1  + |x|\exp|x|$ which holds for all $x\in \bbR$, together with Cauchy-Schwarz and obtain
\begin{equation}
\label{eq:3.21}
\begin{aligned}
&\bbE\left[	\ind_{A_{n,i}^\complement}	\exp\left(	\frac{1}{(np_n)^2} \sum_{j,k} \hat{\xi}_{i,j} 	\hat{\xi}_{i,k} d_{ijk} 	\right)	\right] \le\\
&1 +	\bbE\left[	\ind_{A_{n,i}^\complement}	\left|	\frac{1}{(np_n)^2} \sum_{j,k} \hat{\xi}_{i,j}	\hat{\xi}_{i,k} d_{ijk}	\right|	\exp\left(	\left|
\frac{1}{(np_n)^2} \sum_{j,k} \hat{\xi}_{i,j}	\hat{\xi}_{i,k} d_{ijk}	\right|	\right)	\right] \le \\
&1 +	\bbE\left[	\left(	\frac{1}{(np_n)^2} \sum_{j,k} \hat{\xi}_{i,j}	\hat{\xi}_{i,k} d_{ijk}	\right)^2
\right]^{1/2}	\bbE\left[	\ind_{A_{n,i}^\complement}	\exp\left(	\frac{2}{(np_n)^2} \sum_{j,k} \hat{\xi}_{i,j}	\hat{\xi}_{i,k} d_{ijk}	\right) \right]^{1/2}\, .
\end{aligned}
\end{equation}
Under the condition	that we are in $A_{n,i}^\complement$, it holds that
\begin{equation}
\left|\frac{2}{(np_n)^2} \sum_{j,k} \hat{\xi}_{i,j} \hat{\xi}_{i,k} d_{ijk}	\right|	\le 2   K^2 d_\star
\end{equation}
so that the exponential expectation can be bounded as $\exp{\left\{2 K^2 d_\star\right\}}$. Estimating the moment expectation leads to
\begin{align}
	\bbE \left[ \left(	\frac{1}{(np_n)^2} \sum_{j,k} \hat{\xi}_{i,j}	\hat{\xi}_{i,k} d_{ijk}	\right)^2 \right] = \frac{1}{(np_n)^4} \sum_{j,k,p,q} \bbE \left[\hat{\xi}_{i,j}	\hat{\xi}_{i,k} \hat{\xi}_{i,p} \hat{\xi}_{i,q}\right] d_{ijk} d_{ipq} \le \\
	\le \frac{d^2_\star}{(np_n)^4} \left[ np_n + 3(np_n)^2 \right] \le \frac{4d^2_\star}{(np_n)^2}.
\end{align}
From \eqref{eq:3.21}, we get
\begin{equation}
\bbE\left[	\ind_{A_{n,i}^\complement}	\exp\left(	\frac{1}{(np_n)^2} \sum_{j,k} \hat{\xi}_{i,j} 	\hat{\xi}_{i,k} d_{ijk} 	\right)	\right] \le
1 + \frac{2 d_\star}{np_n}\exp{\left\{2   K^2 d_\star\right\}}.
\end{equation}
Putting everything back in \eqref{eq:2fLLN}, one obtains
\begin{equation}
\bbE \left[\ind_{A_n^\complement}\exp\left(\frac{1}{(np_n)^2}\sum_{i,j,k=1}^n \hat{\xi}_{i,j} \hat{\xi}_{i,k} d_{ijk}\right)\right] \le \exp\left\{ \frac{2d_\star}{p_n}\exp{\left\{2   K^2 d_\star\right\}}\right\}\, ,
\end{equation}
which gives \eqref{goalLLN}.
\end{proof}

\bigskip
We are now ready for

\begin{proposition}
	\label{th:wips}
	If \eqref{eq:p_n-1} holds, then for all $T>0$,
	\begin{equation}
	\lim_{n\to \infty}   \frac{1}{n} \sum_{i=1}^n \sup_{t\in[0,T]} \left| \theta^{i,n}_t - \bar{\theta}^{i,n}_t\right|^2 = 0, \quad  \bP \otimes \bbP\text{-a.s.}.
	\end{equation}
\end{proposition}

\begin{proof}
	For $i \in \{1,\dots,n\}$, consider
	\begin{multline}
	 \left| \theta^{i,n}_t - \bar{\theta}^{i,n}_t\right|^2 = \\
	  2 \int_{0}^{t} \left(\theta^{i,n}_s - \bar{\theta}^{i,n}_s\right) \left( F\left(\theta^{i,n}_s\right) - F\left(\bar{\theta}^{i,n}_s\right) + \frac{1}{n} \sum_{j=1}^{n}\left[\frac{\xi_{i,j}^{(n)}}{p_{n}}\Gamma\left(\theta_s^{i,n},\theta_s^{j,n}\right) - \Gamma\left(\bar\theta_s^{i,n},\bar\theta_s^{j,n}\right) \right]\right) \dd s \le \\
	 2 L_F \int_{0}^{t} \left| \theta^{i,n}_s - \bar{\theta}^{i,n}_s \right|^2 \dd s +
	2 L_\Gamma  \frac{1}{n} \sum_{j=1}^{n} \frac{\xi_{i,j}^{(n)}}{p_{n}} \int_{0}^{t} \left(\left|\theta_s^{i,n} - \bar{\theta}_s^{i,n} \right| + \left|\theta_s^{j,n} - \bar{\theta}_s^{j,n} \right|\right) \left|\theta_s^{i,n} - \bar{\theta}_s^{i,n} \right| \dd s \\
	+ 2 \int_{0}^{t} \left| \frac{1}{n} \sum_{j=1}^{n} \left(\frac{\xi_{i,j}^{(n)}}{p_{n}} - 1\right) \Gamma\left(\bar\theta_s^{i,n},\bar\theta_s^{j,n}\right) \right| \left|\theta_s^{i,n} - \bar{\theta}_s^{i,n} \right| \dd s,
	\end{multline}
	which gives
	\begin{multline}
	 \left| \theta^{i,n}_t - \bar{\theta}^{i,n}_t\right|^2 \le \\
	(2 L_F +1)\int_{0}^{t} \left| \theta^{i,n}_s - \bar{\theta}^{i,n}_s \right|^2 \dd s +
	L_\Gamma  \frac{1}{n} \sum_{j=1}^{n} \frac{\xi_{i,j}^{(n)}}{p_{n}} \int_{0}^{t} \left[ 3 \left|\theta_s^{i,n} - \bar{\theta}_s^{i,n} \right|^2 + \left|\theta_s^{j,n} - \bar{\theta}_s^{j,n} \right|^2 \right]\dd s \\
	+  \int_{0}^{t} \left| \frac{1}{n} \sum_{j=1}^{n} \left(\frac{\xi_{i,j}^{(n)}}{p_{n}} - 1\right) \Gamma\left(\bar\theta_s^{i,n},\bar\theta_s^{j,n}\right) \right|^2\dd s.
	\end{multline}
	Summing over $i$ and dividing by $n$, one obtains
	\begin{multline}
	 \frac 1n \sum_{i=1}^{n} \left| \theta^{i,n}_t - \bar{\theta}^{i,n}_t\right|^2 \le \\
	\le  \left(2 L_F +1 + L_\Gamma \sup_{i=1,\dots,n} \sum_{j=1}^{n} \frac{3\xi_{i,j}^{(n)}+\xi_{j,i}^{(n)}}{np_{n}}\right) \int_{0}^{t} \frac 1n \sum_{i=1}^{n} \left| \theta^{i,n}_s - \bar{\theta}^{i,n}_s \right|^2 \dd s
	\\
	+ \frac 1n \sum_{i=1}^{n} \int_{0}^{t} \left| \frac{1}{n} \sum_{j=1}^{n} 	\left(\frac{\xi_{i,j}^{(n)}}{p_{n}} - 1\right) \Gamma\left(\bar\theta_s^{i,n},\bar\theta_s^{j,n}\right) \right|^2\dd s.
	\end{multline}
	In order to bound $ \sum_{j=1}^{n} \frac{3\xi_{i,j}^{(n)}+\xi_{j,i}^{(n)}}{np_{n}}$ for all $i=1,\dots,n$, we
	choose $K>K_{C}$ and use Lemma \ref{lem:max-degree} to obtain that
	\begin{equation}
	\label{goal:max-degree}
	\sup_{i=1,\dots,n} \sum_{j=1}^{n} \frac{3\xi_{i,j}^{(n)}+\xi_{j,i}^{(n)}}{np_{n}} \le 4+ 4K\, ,
	\end{equation}
	 $\bbP (\dd \xi)$-a.s..
	The application of Gronwall lemma to
	\begin{equation}
	S_n(t) = \frac 1n \sum_{i=1}^{n} \left| \theta^{i,n}_t - \bar{\theta}^{i,n}_t\right|^2,
	\end{equation}
	leads to
	\begin{equation}
	\label{eq:gron}
	S_n(t) \le \int_0^t \exp{\left\{G (t-s)\right\}} \left( \frac 1n \sum_{i=1}^{n} \Delta_i (s)\right)\dd s,
	\end{equation}
	with $G= 2L_F +1 + (4+4K)L_\Gamma >0$ and $\Delta_i (s)$ defined in \eqref{eq:delta_i}.
	%\begin{equation}
	%\Delta_i (s) = \bE \left| \frac{1}{n} \sum_{j=1}^{n} 	\left(\frac{\xi_{i,j}^{(n)}}{p_{n}} - 1\right) \Gamma\left(\bar\theta_s^{i,n},\bar\theta_s^{j,n}\right) \right|^2.
	%\end{equation}
	Therefore
	\begin{equation}
	\label{eq:gron-S}
	\sup_{t\in[0,T]} S_n(t) \le \exp{\left\{GT\right\}} \int_0^T \frac 1n \sum_{i=1}^{n} \Delta_i (s)\dd s.
	\end{equation}
	The last estimate is true for all realizations of the Brownian motions. Taking the limit for $n$ which tends to $\infty$ and integrating the RHS of \eqref{eq:gron-S}, first with respect to $\bbP$ (recall Lemma \ref{lem:exp}), completes the proof of Proposition~\ref{th:wips}.
\end{proof}

\bigskip
\noindent
\emph{Proof of Theorem~\ref{th:LLN0}.}
Since we already know that $\bar{\mu}^n_\cdot$ converges $\bP$-a.s. to $\mu_\cdot$ in $C^0([0,T]; \cP(\R))$ (see \cite[Theorem 1.6]{cf:G88}), it suffices to show that
\begin{equation}
\lim_{n\to\infty} \sup_{0\le t \le T} \dd_{\text{bL}}\left( \mu^n_t,\bar{\mu}^n_t\right) = 0, \quad \bP\otimes\bbP\text{-a.s..}
\end{equation}
For every $f:\R\to[0,1]$ 1-Lipschitz function, we have
\begin{equation}
\left|\int_{\R} f(\theta) \left( \mu_t^n - \bar{\mu}^n_t\right)(\dd\theta)\right| = \left|\tfrac{1}{n} \sum_{i=1}^n f(\theta^{i,n}_t) - f(\bar{\theta}^{i,n}_t)\right| \le \tfrac{1}{n} \sum_{i=1}^n \left|\theta^{i,n}_t - \bar{\theta}^{i,n}_t\right|.
\end{equation}
In particular,
\begin{equation}
\sup_{0\le t \le T} \dd_{\text{bL}}\left( \mu^n_t,\bar{\mu}^n_t\right) \le \sqrt{\frac 1n \sum_{i=1}^n \sup_{0\le t\le T}\left|\theta^{i,n}_t - \bar{\theta}^{i,n}_t\right|^2}.
\end{equation}
The proof follows from Proposition \ref{th:wips}.
\qed

\bigskip

\section{Large Deviation results: proofs.}
\label{sec:LDP}
The proof of Theorem~\ref{thm: LDP} relies on two results that contain most of the work.
We first prove Theorem~\ref{thm: LDP} assuming these two results, and prove them right after.

\bigskip

\noindent
\emph{Proof of Theorem~\ref{thm: LDP}.}
	Observe that we can write  (\ref{eq:dRG}) as
	\begin{equation}
		\label{eq:dRG2}
		\dd\theta_{t}^{i,n}\, =\, F\left(\theta_{t}^{i,n}\right)\dd 	t+\frac{1}{n}\sum_{j=1}^{n}\Gamma\left(\theta_{t}^{i,n},\theta_{t}^{j,n}\right)\dd t
		+\sigma \left(\theta_{t}^{i,n}\right) c_i\left( \theta_{t}^{n}\right) \dd t +
		\sigma \left(\theta_{t}^{i,n}\right) \dd B_{t}^{i},
	\end{equation}
	with
	\begin{equation}
	\label{eq:ci}
		c_i\left( \theta_{t}^{n}\right)\, :=\, \frac 1 {n \sigma \left(\theta_{t}^{i,n}\right)} 	{\sum_{j=1}^{n}\left(\frac{\xi_{i,j}^{(n)}}{p_{n}}-1\right)\Gamma\left(\theta_{t}^{i,n},\theta_{t}^{j,n}\right)}\, .
	\end{equation}
	Recall that $\bP_n^\xi$, respectively $\bP_n$, is the law of the trajectories $\{ \theta_{t}^{i,n}\}_{i=1, \ldots, n; \, t\in [0, T]}$,
	respectively  the law of $\{ \bar \theta_{t}^{i,n}\}_{i=1, \ldots, n; \, t\in [0, T]}$.
	The Radon-Nikodym derivative $\dd \bP_n^\xi /\dd \bP_n$ is $\exp(M_T^n-\langle  M^n \rangle_T/2)$ with
	\begin{equation}
	\label{eq:MT}
		M_T^n\, =\, \sum_{i=1}^n\int_0^T c_i(\theta^n_t) \dd \theta^{i,n}_t\ \ \textrm{ and } \ \
		\langle  M^n \rangle_T\, =\, \sum_{i=1}^n\int_0^T c^2_i(\theta^n_t) \dd t\, .
	\end{equation}

The following lemma is given for every realization of  $\xi^{(n)}$ and it has a deterministic nature.  Recall that $\chi=  \cP \left( C^0([0,T]; \bbR) \right)$ and $\overline{L}_n$ defined in \eqref{eq:L}. Then $\overline P_n(\cdot):= \bP( \overline{L}_n \in \cdot)$ is the law of the empirical process associated to \eqref{eq:dnonRG} and $P^\xi_n(\cdot):= \bP( L_n \in \cdot)$ is the one associated to \eqref{eq:dRG},  $\overline P_n$ and $P^\xi_n$ are probabilities on $\chi$.

\medskip

\begin{lemma}
	\label{lem:1}
	Suppose $\overline{P}_{n}\left( \cdot \right) $ satisfies a LDP on $\chi$
	with rate function $I$. If, for every  $C\in \bbR$,
\begin{equation}
\label{condLem}
	\lim_{n\rightarrow \infty} \frac 1n \log \bE_{n} \left[\exp\left\{C \langle  M^n \rangle_T \right\} \right] = 0,
	\end{equation}
	then $P^{\xi}_{n}\left( \cdot \right)$ satisfies a LDP on $\chi$ with the same rate function as $\overline{P}_{n}$.
\end{lemma}

\medskip

	Since we want the LDP to hold $\bbP(\dd \xi)$-a.s., we need to show that condition \eqref{condLem} holds in this sense.
To this aim, we redefine the sets $\Omega_n$ given in \eqref{eq:OmeganLLN}, as
\begin{equation}
\label{eq:Omegan}
	\Omega_n \, :=\, \left\{ \xi : \frac 1n \log \bE_{n} \left[\exp\left\{C_n \langle  M^n \rangle_T \right\} \right] > \delta_n  \right\}\, ,
\end{equation}
where $\gd_n$ and $1/C_n$ tend to zero: they have to do so in a slow way and arbitrarily slow will do for us (explicit choices will be given at the end of the proof).

% it turns out to be more practical to make them go to zero, but they have to go to zero in a sufficiently slow
%fashion and arbitrarily slow suffices. More precisely, set
%\begin{equation}
%\label{eq:eta}
% \eta(n)\, :=\, \frac1{p^2_n n \log n} \, ,
%\end{equation}
%so $\lim_n\eta(n)=0$
%(recall \eqref{eq:hyp-p}):
% in the remainder we will assume
%\begin{equation}
%\label{eq:pseudoconstants}
%C_n \, =\, \frac 1{\sqrt{\eta(n)}}  =\, p_n \sqrt{n \log n} \, %, \ \ \ C_n\, =\, o(p_n n)\,  \
%\ \text{ and } \  \frac1{\gd_n} \, =\, o\left(\frac{p_n n}{C_n}\right)\, .
%\end{equation}
%To be even more concrete,  choose $\gd_n= 1/n^b$, for any $b \in (0, 1/2)$.

 We need also
 \begin{equation}
 \label{eq:Omegastar}
	\Omega^*\, =\, \left\{\xi:\text{ there exists } n_0  \text{ s.t. } \frac 1n \log \bE_{n} \left[\exp\left\{C_n \langle  M^n \rangle_T \right\} \right] \, \le\,  \delta_n \text{ for every } n\ge n_0 \right\} .
	\end{equation}

\medskip

\begin{lemma}
\label{lem:BC}
	Assuming \eqref{eq:p_n-1} we have  that  $\bbP \left( \Omega^* \right) = 1$.
\end{lemma}

\medskip

One readily sees that
Lemma~\ref{lem:BC} provides the missing ingredient and the proof of Theorem~\ref{thm: LDP} is complete.
\qed

\medskip

\noindent
\emph{Proof of Lemma~\ref{lem:1}.}
	%Since the generators of dynamics (\ref{eq:dRG2}) and (\ref{eq:dnonRG}) only differ w.r.t. the drift vector, $\bP^{\xi}_{n}$ is absolutely continuous with respect to $\bP_{n}$ and Girsanov formula gives
	%\begin{equation}\frac{\dd \bP_n^\xi}{\dd \bP_n} = \exp(M_T^n-\langle  M^n \rangle_T/2)\, ,	\end{equation}
%	and we recall that $M_T^n$ is defined in  \eqref{eq:MT} and $\langle \, \cdot \, \rangle$ denotes the quadratic	variation.
	Recall \eqref{eq:dRG2}-\eqref{eq:MT}.
	We have  to
	show that \eqref{eq:LDP} holds. Consider $A$ a measurable set and recall that $A^\circ$ is the interior of $A$ and $ \bar{A}$ is its closure.

	Let $p,q>1$ such that $ \tfrac{1}{p}+\tfrac{1}{q}=1$. Then
	\begin{equation}
	P^{\xi}_{n}\left(A^\circ\right) =\bP^{\xi}_{n}\left(\left\{\mu^n_t\right\}_{t\in [0,T]}\in A^\circ\right) =\bE_{n}\left[1_{\left\{ \left\{\mu^n_t\right\}_{t\in [0,T]} \in A^\circ\right\} }\exp\left\{ M_{T}^{n}-\tfrac{1}{2}\left\langle M^{n}\right\rangle _{T}\right\} \right]
	\end{equation}
	and H\"older inequality gives
	\begin{equation}
	P^{\xi}_{n}\left(A^\circ\right) \ge \left(  \overline{P}_{n}\left(A^\circ\right) \right)^p \left(\bE_n \left[ \exp \left\{-\frac qp M^n_T -\frac 12 \frac qp \langle M^n \rangle_T \right\} \right] \right)^{- \frac pq}.
	\end{equation}
	Now observe that Cauchy-Schwarz inequality together with the fact that an exponential martingale has expectation less or equal to $1$ (see \cite[Theorem 5.2]{cf:IW}) imply
	\begin{equation}
	\label{exp1}
	\bE_n \left[ \exp \left\{-\frac qp M^n_T -\frac 12 \frac qp \langle M^n \rangle_T \right\} \right] \le \bE_n \left[ \exp \left\{ \left( \tfrac{2q^2}{p^2} + \frac qp \right) \langle M^n \rangle_T \right\} \right] ^{\frac 12}.
	\end{equation}
	Hence
	\begin{equation}
	P^{\xi}_{n}\left(A^\circ\right) \ge \left(  \overline{P}_{n}\left(A^\circ\right) \right)^p \left(\bE_n \left[ \exp \left\{  \left( \tfrac{2q^2}{p^2} - \frac qp \right) \langle M^n \rangle_T \right\} \right] \right)^{- \frac p{2q}}.
	\end{equation}

	In particular, one obtains
	\begin{align}
	\liminf_{n\rightarrow\infty}\tfrac{1}{n}\log P^{\xi}_{n}\left(A^\circ\right)\ge
	- p \inf_{x\in A^\circ}I(x)- \tfrac{p}{2q}	\liminf_{n\rightarrow \infty} \frac 1n \log \bE_{n} \left[\exp\left\{C \langle  M^n \rangle_T \right\} \right],
	\end{align}
	with $C= \left( \tfrac{2q^2}{p^2} - \frac qp \right) $. By hypothesis the second term on the right is zero and since
	\begin{equation}
	\liminf_{n\rightarrow\infty}\tfrac{1}{n}\log P^{\xi}_{n}\left(A^\circ\right)\ge -p \inf_{x\in A^\circ}I(x),
	\end{equation}
	is true for all $p> 1$, the lower bound in \eqref{eq:LDP} is established.

	The upper bound is almost the same: let $p,q>1$  be such that $ \tfrac{1}{p}+\tfrac{1}{q}=1 $. Similarly
	\begin{equation}
	P^{\xi}_{n}\left(\bar{A}\right) \le \left(  \overline{P}_{n}\left(\bar{A}\right) \right)^{\frac 1p} \left(\bE_n \left[ \exp \left\{q M^n_T -\frac 12 q \langle M^n \rangle_T \right\} \right] \right)^{\frac 1q},
	\end{equation}
	and, using the properties of exponential martingales as in (\ref{exp1}), one gets
	\begin{equation}
	P^{\xi}_{n}\left(\bar{A}\right) \le \left(  \overline{P}_{n}\left(\bar{A}\right) \right)^{\frac 1p} \left(\bE_n \left[ \exp \left\{(2q^2 - q) \langle M^n \rangle_T \right\} \right] \right)^{\frac 1{2q}}.
	\end{equation}
	Finally, the desired inequality reads
	\begin{equation}
	\limsup_{n\rightarrow\infty}\tfrac{1}{n}\log P^{\xi}_{n}\left(\bar{A}\right)\le -\frac 1p \inf_{x\in \bar{A}}I(x) + \frac 1{2p}	\liminf_{n\rightarrow \infty} \frac 1n \log \bE_{n} \left[\exp\left\{C \langle  M^n \rangle_T \right\} \right],
	\end{equation}
	with $C=2q^2 - q$. And we conclude as before.
	\qed

\bigskip

\noindent
\emph{Proof of Lemma~\ref{lem:BC}.}
	We want to show that
	\begin{equation}
	\sum_{n \in \bbN} \bbP \left( \Omega_n \right) < \infty.
	\end{equation}
	As in the proof of Lemma \ref{lem:exp}, let $K>2$ and consider the events $A_n$ defined in \eqref{eq:An},
	\begin{equation*}
	A_n = \bigcup_{i=1}^{n} A_{n,i} \ \
	\text{ with } \ \
	A_{n,i} = \left\{ \xi^{(n)} : \sum_{j} \left|\frac{\xi^{(n)}_{i,j}}{p_n} - 1\right| > Kn.
	\right\}.
	\end{equation*}
	Following the proof of Lemma \ref{lem:max-degree}, we use $\bbP \left( \Omega_n \right) \le \bbP \left( \Omega_n \cap A_n^\complement \right) + \bbP \left( A_n \right)$ and \eqref{eq:An-con}, i.e. $\sum_{n\in\N} \bbP(A_n) < \infty$, so that one is left with proving that $\sum_{n \in \bbN} \bbP \left( \Omega_n \cap A_n^\complement \right) < \infty$.

	By Markov's inequality applied to $\bbP\left(\, \cdot \, \vert  A_n^\complement \right)$ we see that
	\begin{equation}
	\bbP(\Omega_n \cap A_n^\complement)
	\le \exp\left(-n \delta_n + \log \bE \bbE[\ind_{A_n^\complement}\exp(C_n \langle{M^n}\rangle_T)]\right).
	\end{equation}
	so it suffices to show that
	\begin{equation}
	\label{goal}
	 \log \bE \bbE[\ind_{A_n^\complement}\exp(C_n \langle{M^n}\rangle_T)] \, =\, o(n \gd_n)\,.
	\end{equation}

	 To lighten the notation we go back to using the centered random variables
	 $\hat\xi_{i,j} := \xi^{\left( n \right)}_{i,j} - p_n$ (cf. \eqref{eq:notations}). With these notations,  $\langle M^n \rangle_{T}$ can be rewritten as
	\begin{equation}
	\langle M^n \rangle_{T} \, =\, \frac{1}{( p_n n)^{2}}\sum_{i,j,k=1}^{n}\hat\xi_{i,j}\hat \xi_{i,k} c_{ijk},
	\end{equation}
	where
	\begin{equation}
	c_{ijk} \, =\, \int_0^T \frac{1}{ \sigma^{2} \left( \theta^{i,n}_t \right)} \Gamma\left(\theta_{t}^{i,n},\theta_{t}^{j,n}\right) \Gamma\left(\theta_{t}^{i,n},\theta_{t}^{k,n}\right) \dd t.
	\end{equation}
	Observe that $| c_{ijk}| \le c_\star$ given the boundness of $\Gamma$ and the conditions on $\sigma$.

	The estimation of \eqref{goal} is exactly the same as in \eqref{goalLLN}, where $d_{ijk}$ are replaced by $C_n c_{ijk}$ (and $d_\star$ by $C_n c_\star$). Following the same strategy, we get
	\begin{equation}
	\prod_i \bbE\left[	\ind_{A_{n,i}^\complement}	\exp\left(	\frac{C_n}{(np_n)^2} \sum_{j,k} \hat{\xi}_{i,j} 	\hat{\xi}_{i,k} c_{ijk} 	\right)	\right] \le
	\left(1 + \frac{2 C_n}{np_n}\exp{\left\{2 C_n  K^2 c_\star\right\}}\right)^n.
	\end{equation}
	Therefore
	\begin{equation}
	\log \bE \bbE \left[ \ind_{A_n^\complement} \exp \left( C_n \langle M^n \rangle_T \right)\right] \le \frac{2 C_n}{p_n}\exp{\left\{2 C_n  K^2 c_\star\right\}}.
	\end{equation}
	Which gives \eqref{goal} when $C_n=o(\log(np_n))$ and $\frac 1{\delta_n} = o\left(\frac{np_n}{\exp\{c \, C_n\}}\right)$ with $c>2 K^2 c_\star$:  choose, for example, $C_n =\sqrt{\log(np_n)}$ and $\delta_n = \tfrac{1}{\sqrt{np_n}}$.
	\qed

\section*{Acknowledgments}
G.G. thanks Amir Dembo for an insightful discussion. F.C. acknowledges the support from the European Union’s Horizon 2020 research and innovation programme under the Marie Sk\l odowska-Curie grant agreement No 665850.
 H.D. acknowledges support
of  Sorbonne Paris Cit\'e, in the framework of the ``Investissements d'Avenir", convention ANR-11-IDEX-0005, 	and of
the People Programme (Marie Curie Actions) of the European Unions Seventh Framework Programme (FP7/2007-2013) under REA grant agreement n. PCOFUND-GA-2013-609102, through the PRESTIGE programme coordinated by Campus France. G.G.  acknowledges the support of grant ANR-15-CE40-0020.

\end{document}